\documentclass[11pt]{amsart}
 
\usepackage[cm]{fullpage}

\usepackage{latexsym}
\usepackage{amsfonts}
\usepackage{amsmath}
\usepackage{amssymb}
\usepackage{tikz-cd}
\usepackage{enumitem}
\usepackage{hyperref}
\usepackage{graphicx} 
\graphicspath{ {./images/} }
\usepackage{comment}

\usepackage{amsmath,amscd, float,rotating}
\usepackage{pb-diagram}

\numberwithin{equation}{section}

\def\XXint#1#2#3{{\setbox0=\hbox{$#1{#2#3}{\int}$}
\vcenter{\hbox{$#2#3$}}\kern-.5\wd0}}

\def\na{\nabla}

\newcommand{\al}{\alpha}
\newcommand{\be}{\beta}

\newcommand{\cU}{\mathcal{U}}

\newcommand{\dbar}{\overline{\partial}}

\newcommand{\ddbar}{\sqrt{-1}\partial\dbar}

\newtheorem{theorem}{Theorem}[section]
\newtheorem{proposition}{Proposition}[section]
\newtheorem{lemma}{Lemma}[section]
\newtheorem{example}{Example}[section]

\newtheorem{definition}{Definition}[section]
\newtheorem{corollary}{Corollary}[section]
\newtheorem{remark}{Remark}[section]

\begin{document}

~

\title{Dirichlet Problem for Complex Monge-Amp\`ere equation near an isolated Klt singularity}

\author{Xin Fu}



\address{Department of Mathematics,  University of California, Irvine, CA 92617}

\email{fux6@uci.edu}



\begin{abstract}
We solve the Dirichlet Problem for complex Monge-Amp\`ere equation near an isolate Klt singularity, which generalizes the result of Eyssidieux-Guedj-Zeriahi \cite{EGZ}, where the Monge-Amp\`ere equation is solved on singular varieties without boundary. As a corollary, we construct solutions to Monge-Amp\`ere equation with isolated singularity on strongly pseudoconvex domain $\Omega$ contained in $\mathbb C^n$.  
\end{abstract}
\bigskip

\maketitle

\bigskip


\section{Introduction}\label{intro}
Complex Monge-Amp\`ere equations are a fundamental tool to study K\"ahler geometry and, in particular, canonical K\"ahler metrics of Einstein type on smooth and singular K\"ahler varieties. Yau's solution to the Calabi conjecture establishes the existence of Ricci flat K\"ahler metrics  on K\"ahler manifolds of vanishing first Chern class by a priori estimates for complex Monge-Amp\`ere equations \cite{Y1}. 
\begin{theorem}[Yau \cite{Y1}]\label{MA}Let $(X, \omega)$ be a K\"ahler manifold of complex dimension $n$ equipped with a K\"ahler metric $\omega$. 
We consider the following complex Monge-Amp\`ere equation 
\begin{equation} \label{cma} 
(\omega+ \ddbar \varphi)^n = e^{ - f} \omega^n,
\end{equation}
where  $f\in C^\infty(X)$ satisfies the normalization condition $$\int_X e^{-f}\omega^n = \int_X \omega^n = [\omega]^n. $$
\end{theorem}
In the deep work of Kolodziej \cite{Kol1}, Yau's $C^0$-estimate for solutions of equation (\ref{cma}) is tremendously improved by applying the pluripotential theory and it has important applications for singular and degenerate geometric complex Monge-Amp\`ere equations.  
\begin{theorem}[Kolodziej\cite{Kol1}] Suppose  the right hand side of equation (\ref{cma}) satisfies the following $L^p$ bound
$$\int_{X} e^{-pf} \omega^n \leq K, ~ \textnormal{for ~some}~ p>1, $$
then there exists  $C=C(X, \omega, p, K)>0$ such that any solution $\varphi$ of equation (\ref{cma}) satisfies the the following $L^\infty$-estimate
$$ \| \varphi - \sup_X \varphi \|_{L^\infty(X)} \leq C. $$
\end{theorem}
By using Kolodziej's fundamental estimate,  K\"ahler-Einstein metric on  canonical polarized variety with Klt singularity is constructed in \cite{EGZ}. More precisely, the following holds: 

\begin{theorem}[Eyssidieux-Guedj-Zariahi \cite{EGZ}, Zhang \cite{Z}]\label{EGZ} Suppose $X$ is canonical polarized or Calabi-Yau variety with Klt singularity, then there exists a solution $\varphi \in \textnormal{PSH}(X,\omega)\cap C^{\infty}(X_{reg})\cap L^{\infty}_{loc}(X_{reg})$ on the regular locus of $X$ to the equation (\ref{MA}).  
\end{theorem}
\begin{remark} Choosing suitable right hand side of equation (\ref{MA}), the metric induced by $\omega+\ddbar\varphi$ is a
K\"ahler-Einstein metric with negative scalar curvature or Ricci flat metric on $X$.
\end{remark}
On the other hand, Dirichlet problem for complex Monge-Amp\`ere equation is studied in the fundamental work of Bedford-Taylor \cite{BT1,BT} and also Caffarelli- Kohn-Nirenberg-Spruck \cite{CKNS}. (see also \cite{Guan,CXX}). Therefore it is a natural question to study the Dirichlet problem for complex Monge-Amp\`ere equations near an isolated Klt singularity from the PDE point of view.

Another motivation to study the Dirichlet problem for complex Monge-Amp\`ere equations near an isolated Klt singularity is to understand the tangent cone of K\"ahler-Einstein metrics proposed by Donaldson-Sun in \cite{DS2}. For example, in the discussion section of \cite[Page 116]{H-S}, it is proposed to study the local Ricci flat metric near an isolated Klt singularity and show that the metric tangent cone only depends on the algebraic structure of the singularity. In this short note, we first show the existence of local Ricci flat metrics by solving Dirichlet problem for complex Monge-Amp\`ere equation. 

Now we proceed to set up the equation. Let $(X,p)$ be an affine variety embedded in $(\mathbb C^n,0)$ with $p$  an isolated singular point of $X$  and also let $\rho$ be a  nonnegative smooth plurisubharmonic (PSH) function on $\mathbb C^n$ with $\rho(0)=0$ and $\ddbar \rho>0$.  We let  $\mathcal U \subset\subset X$ be the open domain defined by
\begin{equation}\label{domain1}
\mathcal U=\{\rho< a\} \cap X,
\end{equation}
for some sufficiently small constant $a>0$. 
In this note, we  fix $$\rho = \log(1+\sum_{j=1}^n |z_j|^2).$$ By choosing a  generic sufficiently small $a>0$, we can assume that $\partial\mathcal U$ is smooth and strongly pseudoconvex. We will fix such a domain $\mathcal U$ for $(X, p)$ for the rest of the paper. If the canonical divisor $K_{\mathcal U}$ is Cartier,  we will fix a local volume measure $\Omega_X$ on $\mathcal U$ by 
$$\Omega_X=  (\sqrt{-1})^n\nu \wedge \overline \nu, $$ 
where $\nu$ is a local generator of $K_{\mathcal U}$. Similarly, we can define $\Omega_X$ when $K_{\mathcal U}$ is only $\mathbb{Q}$-Cartier. 

We will consider the following Dirichlet problem for a complex Monge-Amp\`ere equation related to the K\"ahler-Einstein metric on $\mathcal U$.
\begin{equation} \label{ocmp1}
\left\{
\begin{array}{ll}
(\ddbar \varphi)^n =  \Omega_X,  \\
& \\ 
\varphi|_{\partial \cU} =\psi,\\
\end{array}\right. 
\end{equation}
where $\psi\in C^\infty(\partial \mathcal U)$. Now we can state the main result of the paper. 
\begin{theorem}\label{main1} 
Let $(\mathcal U,p)$ be a germ of an isolated Klt singularity defined as in definition \ref{domain1}. There exists a solution $\varphi \in \textnormal{PSH}(\mathcal U)\cap C^{\infty}(\overline{\cU}\setminus \{p\})\cap L^{\infty}_{loc}(\mathcal U\setminus \{p\})$ of equation \eqref{ocmp1}. Let $g_{KE}$ be the K\"ahler metric defined by $\ddbar \varphi$ in $\mathcal U$. Then it satisfies the K\"ahler-Einstein equation
\begin{equation}\label{keq}
Ric(g_{KE})= 0
\end{equation}
 in $\mathcal U$.

\end{theorem}
As a corollary, we obtain singular solution to complex Monge-Ampe\`ere equation on a strongly pseudoconvex domain.
\begin{corollary}\label{Singularsolution}
Let $\Omega$ be a strongly pseudoconvex domain  in $\mathbb C^n$ with smooth boundary and $V:=dz_1\wedge dz_2 \ldots\wedge dz_n$ be a holomorphic volume form. Suppose origin $\{o\}$ is contained in $\Omega$, then for any $s>0$,  $f\in C^{\infty}(\Omega)$, $\psi\in C^{\infty}(\partial\Omega),$ and $\lambda=0$ or $1$, the following Monge-Amp\`ere  equation 
 \begin{equation}\label{Singular}
\left\{
\begin{array}{rcl}
&&(\ddbar \varphi)^n =e^{f+\lambda\varphi}V\wedge\bar V \\
&&\varphi|_{\partial\Omega}=\psi
\end{array}\right.
\end{equation}
admits a singular solution $\varphi\in C^{\infty}(\Omega\setminus \{o\})$  satisfying $$|\varphi-s\log(\sum_{i=1}^n|z_i|^2)|=\mathcal O(1).$$
\end{corollary}

\section{Preliminaries On Algebraic Singularity}
In this section, we introduce some standard definitions on singularity theory from algebraic geometry.
\begin{definition} \label{sing} Let $X$ be a normal projective variety such that $K_X$ is a $\mathbb{Q}$-Cartier divisor. Let $ \pi :  Y \rightarrow X$ be a log resolution and $\{E_i\}_{i=1}^p$ the irreducible components of the exceptional locus $Exc(\pi)$ of $\pi$. There there exists a unique collection $a_i\in \mathbb{Q}$ such that 
$$K_Y = \pi^* K_X + \sum_{i=1}^{ p } a_i E_i .$$ Then $X$ is said to have log-canonical (resp.\ klt) singularities if $$a_i\geq -1~ \text{(resp.\ $a_i>-1$) for all $i$}. $$
%





\end{definition}

\begin{definition}
$(X,p)$ is said to be  a homogeneous hypersurface singularity if $X$ is an affine variety defined by one homogeneous polynomial $f$. 
\end{definition}

\begin{example} Let $X:=\{\sum_{i=1}^{i=n}z^m_i=0\}$ with $1<m<n$, then $(X,0)$ is a homogeneous hypersurface Klt singularity.
\end{example}

We briefly explain why the above example is a Klt singularity. Fixing a cone resolution $\pi:Y\to X$ of $X$, then there is only one exceptional divisor $E$ which is isomorphic to the projective variety $\tilde X\subset \mathbb{CP}^{n-1}$ defined by $\{\sum_{i=1}^{i=n}z^m_i=0\}$. Now the difference of canonical class can be written as $$K_Y=\pi^*K_X+aE,$$ where $a$ is constant to be determined. By adjunction formula, we have $$K_{\tilde X}=(a+1)E|_{\tilde X}.$$ 
Note that when $m<n$, $-K_{\tilde X}$ is an ample divisor and that the exceptional divisor $E$ when restricted on $\tilde X$ is an anti ample divisor, this implies that $a+1>0$, which is the requirement of Klt singularity. 

Let us also introduce the Kodaira Lemma which is quite useful.
 \begin{lemma}\cite[Lemma 2.62]{KM}\label{lem:kodaira} Let $\pi:Y\to X$ be a birational morphism. Also let $\omega$ be a K\"ahler form on $X$ and $\theta$ be a K\"ahler form on $Y$. Assume that $Y$ is quasi projective and $X$ is $\mathbb Q$ factorial. 
 Then there exists an effective simple normal crossing divisor $D$ supporting on the exceptional locus,  a hermitian metric $h_D$ on the line bundle $D$ and a constant $s_0>0$ such that  $$\pi^*\omega + s\ddbar\log h_D > 0,$$ for all $0<s\leq s_0$. By adjusting the coefficients of $D$ we may assume that $s_0=1$ and that there exists a constant $\be>0$ such that $$\pi^*\omega+\ddbar\log h_D > \be\theta.$$ 
 \end{lemma}
 
\section{Proof of Main Theorem \ref{main1}}
We have to prescribe singularities of the solution $\varphi$ to obtain a canonical K\"ahler-Einstein current on $\cU$, where $\cU$ is part of an affine variety $X$ as in (\ref{domain1}). To do so, we   lift all the data to a log resolution $\pi: Y \rightarrow X$. 
By adjunction formula:
 $$K_Y = \pi^* K_X + \sum_i a_i E_i - \sum_j b_j F_j, ~a_i\geq 0, ~ 0< b_j \leq 1. $$


Let $\sigma_{E}$ be the defining section for $E=\sum_{i=1}^I a_i E_i$ and $\sigma_F$ be the defining section for $F=\sum_{j=1}^J  b_j F_j$ (possibly multivalued). We  equip the line bundles associated to $E$ and $F$ with smooth hermitian metric $h_E$, $h_F$ on $Y$. Recall that domain $\cU$ is defined in (\ref{domain1}). Let $\Omega_Y$ be a smooth strictly positive volume form on $\pi^{-1}(\cU)$,  
defined by 
$$ \Omega_Y  = ( |\sigma_E|^2_{h_E} )^{-1} |\sigma_F|^2_{h_F} \Omega_X. $$
Then lifting equation \eqref{ocmp1} to $\pi^{-1}(\cU)\subset Y$, we have 
\begin{equation}
\left\{
\begin{array}{rcl}
&&(\pi^*\ddbar \varphi)^n =  ( |\sigma_F|^2_{h_F} )^{-1} |\sigma_E|^2_{h_E} \Omega_Y, \\
&&\pi^*\varphi|_{\partial\cU}=\psi.\\
\end{array}\right. 
\end{equation}
Abusing notation, we still denote the domain $\pi^{-1}(\cU)$ by $\cU$.
Let $\theta$ be a fixed smooth K\"ahler form on $Y$ (Note that $Y$ is quasi projective and we may choose $\theta$ as a K\"ahler form on some compactification $\bar Y$ of $Y$.) and we consider the following perturbed family of complex Monge-Amp\`ere equations on $\pi^{-1}(\cU)$ for $s \in (0, 1)$, 
\begin{equation}\label{scmp}
\left\{
\begin{array}{rcl}
&&(s\theta + \ddbar \varphi_s)^n =  (|\sigma_E|^2_{h_E}+ s)  (|\sigma_F|^2_{h_F}+ s)^{-1}\Omega_Y, \\
&&\varphi_s|_{\partial\cU}=\psi.
\end{array}\right.
\end{equation}
 For each $s>0$, we shall first get solution $\varphi_s$ for equation \eqref{scmp}. When $s=0$, equation \eqref{scmp} coincides with equation \eqref{ocmp1}.

Firstly, let us rewrite the equation (\ref{scmp}). 
We choose an arbitrary smooth extension $\psi_1$ of $\psi$, which is supported on a neighborhood of $\partial\cU.$ In particular, $\psi_1$ vanishes on an open neighbourhood of the preimage of the singular point $p$.
Recall that $\rho$ is smooth on $\mathbb C^n$, hence $\pi^*\rho$ is smooth on $\pi^{-1}\cU$. By choosing A  large enough, we define a semipositive 1-1 form $\omega$ on $\cU$ \begin{equation}\label{omega}\omega:=A\ddbar (\rho-a)+\ddbar\psi_1.\end{equation}   Now
$$s\theta+\ddbar\varphi_s=\omega+s\theta+\ddbar(\varphi_s-A(\rho-a)-\psi_1).$$
Let $$\phi_s:=\varphi_s-A(\rho-a)-\psi_1,M=A(\rho-a)+\psi_1,$$ then we can rewrite equations  \eqref{scmp} as a new family of equations with $\phi_s$ as unknown functions and with zero Dirichlet boundary.
\begin{equation}\label{scmp11}
\left\{
\begin{array}{rcl}
&&(\omega + s\theta + \ddbar \varphi_s)^n =  (|\sigma_E|^2_{h_E}+ s)  (|\sigma_F|^2_{h_F}+ s)^{-1}\Omega_Y, \\
&&\varphi_s|_{\partial\cU}=0.
\end{array}\right.
\end{equation}
The above equations hold for $\phi_s$, we abuse notation and still denote the unknown functions by
$\varphi_s$. Notice all the estimates we get for $\phi_s$ also hold for $\varphi_s$ due to the fact that $M$ is independent of $s$. Before we proceed, we do a remark about the $1$-$1$ form $\omega$ in (\ref{omega}) and $\theta$.
\begin{remark}\label{E}
Note that the positive 1-1 form $\ddbar\rho := \ddbar\log(1+\sum_{j=1}^n |z_j|^2)$ defined on $\mathbb {C}^n$ is indeed the restriction of the Fubini-Study metric on $\mathbb {CP}^{n}$, hence we may assume that $\omega  (\theta)$ is the restriction to $\pi^{-1}\cU$ of a closed semi positive (K\"aher) $1$-$1$ form on some projective variety $\bar Y$, where $\bar Y$ is the blow up of $\mathbb {CP}^n$ along some subvariety.  
\end{remark}
From now on, we will focus on equations \eqref{scmp11}.

\begin{lemma}\label{blowup} Let $(X,p)$  be  a homogeneous hypersurface singularity and $\cU$ be defined as in (\ref{domain1}), then for fixed $s>0$, there exists a subsolution to the equation (\ref{scmp11}), i.e., a smooth function $\Phi$ satisfying 
\begin{equation}
\left\{
\begin{array}{rcl}
&&(\omega + s\theta + \ddbar \Phi)^n \geq (|\sigma_E|^2_{h_E}+ s)  (|\sigma_F|^2_{h_F}+ s)^{-1}\Omega_Y, \\
&&\Phi|_{\partial\cU}=0.
\end{array}\right.
\end{equation}
\end{lemma}
\begin{proof}
Let us recall the blowing up of a point construction. 
\begin{equation}\label{blowuppoint}
BL_{o}\mathbb C^n:=\{(z_1,z_2\dots z_n; l_1,l_2,\dots l_n)\in \mathbb C^n \times \mathbb{CP}^{n-1}|z_il_j=z_jl_i\}
\end{equation}On the affine chart $U_i:= \{l_i\neq 0\}$,  we have coordinates $(z_i,u_j=\frac{l_j}{l_i})$, with $j=1,2\dots n$ and $j\neq i$. On $U_i$, the exceptional divisor $E$ is defined by $E:=\{z_i=0\}$. Without lose of generality, we may assume at origin $\{o\}$,
$\omega=\sum_{i=1}^{i=n}\sqrt{-1}dz_i\wedge d\bar{z_i}$. Now consider the pullback  form $\pi^*\omega$ at point $q\in E$,  

\begin{align*}\pi^*\omega&=(1+\sum_{i\neq j}|u_j|^2)\sqrt{-1}dz_i\wedge d\bar z_i+\sqrt{-1}\sum_{i\neq j}|z_i|^2du_j\wedge d\bar u_j\\&+\sqrt{-1}\sum_{i\neq j}u_j\bar z_idz_i\wedge d\bar u_j+\sqrt{-1}\sum_{i\neq j}z_i\bar u_jdu_j\wedge d\bar z_i
\end{align*}
It is clear that $\pi^*\omega$ vanishes along $E$, but it is nonvashing on the normal direction of divisor $E$ sitting in $BL_{o}\mathbb C^n$. More precisely,  we claim that $\pi^*\omega>\frac{1}{2}\sqrt{-1}dz_i\wedge d\bar z_i$, when $|z_i|$ is sufficiently small. The claim will follow from the simple linear algebra fact that the following matrix

\begin{align}\label{eq-gij-new}
  \begin{pmatrix}
	\frac{1}{2}+\sum_{j\neq i}|u_j|^2 & Y\\ \bar Y^T &	|z_i|^2Id_{n-1}  
	\end{pmatrix}
\end{align}
is semipositive, where $Y=(u_1\bar z_i,u_2\bar z_i\cdots,\widehat {u_i\bar z_i}, \cdots, u_n\bar z_i)$.

This shows that we have definite amount of positivity of $\pi^*\omega$ along the normal direction of $E$.
Now since $E$ is compact, for fixed $s$ we can make $\det (A\pi^*\omega+s\theta)$ arbitrary large by choosing a sufficiently large constant $A$. 

For general  homogeneous hypersurface singularity $X:=\{\sum_{i=1}^{n}z_i^m=0\}$ in $\mathbb C^n$, let $X'$ be the strict transformation of $X$ in $BL_{o}\mathbb C^n$ and let $E':=E\cap X'$ be the exceptional divisor, which is a smooth hypersurface in $\mathbb {CP}^n$ defined by $E':=\{\sum_{i=1}^{n}z_i^m=0\}$. Again we consider the form $\pi^*\omega$ near $E'$. Note that the normal bundle of $E'$ sitting in $X'$ is a subbundle of the normal bundle of $E$ sitting in $BL_{o}\mathbb C^n$. Now by the argument in the previous paragraph, we conclude that for fixed $s$ we can make $\det (A\pi^*\omega+s\theta)$ arbitrary large by choosing a large constant $A$.

Now  we construct a subsolution to equation (\ref{scmp}). Recall that $\cU:=\{\rho<a\}$, hence $\phi:=\rho-a$ is equal to $0$ on $\partial\cU$. By choosing $A$ sufficiently large, $A\phi$ will be the desired subsolution for fixed $s>0$. It is well-known that existence of subsolution to equation (\ref{scmp11}) implies existence of solution, see for example \cite{Kol1,Guan}.
 \end{proof}


%
%

\begin{remark}
For general isolated Klt singularity, although it is expected that $\pi^*\omega$ would have certain positivity on the normal direction,  it is not clear to the author how to construct a subsolution to equation (\ref{scmp11}) directly. \end{remark} 

As remarked above, for a fixed $s>0$, we are unable to solve the equation (\ref{scmp11}) for a general Klt singularity by the subsolution approach. However, thanks to the pluripotential theory,  we may obtain uniform a apriori estimates for  $\varphi_s$ which enables us to complete the proof of Theorem \ref{main1} by 2 steps. We remark that the uniform $C^0$ estimate of  Monge-Amp\`ere equations (\ref{scmp11})  have been obtained in many different settings  such as  on  unit ball in \cite{Kol1} and  variety with Klt singularity in \cite{EGZ,D-P,Z}.

\textbf{Step 1:} We first generalize the deep results of \cite{Kol1} to our geometric domain $\cU$ and obtain  uniform $C^0$ bound for $\varphi_s$. The main result in this step is:

\begin{proposition}\label{C0}
 Let $\cU,\theta,\omega$ be the data defined in equation (\ref{scmp}). Assume that $\varphi_s, s<1$ is a smooth solution to the following equation on $\mathcal U$:
 \begin{equation}\label{scmp1}
\left\{
\begin{array}{rcl}
&&(\omega + s\theta + \ddbar \varphi_s)^n = e^f\Omega_{Y}, \\
&&{\varphi_s}|_{\partial\mathcal U}=0,
\end{array}\right.
\end{equation} 
where  $\Omega_Y$ is smooth positive volume on $\mathcal U$ and $f\in C^\infty(\mathcal U)$ satisfying $$\int_{\mathcal U}e^{pf}\Omega_Y\le Q, ~~\textnormal{for some constant} ~~p>1,$$
 then we have $|\varphi_s|\le C$, where $C=C(\cU,\omega,\eta,Q)$.

\end{proposition}
\begin{remark} From equation (\ref{scmp1}), $\Delta_\theta\varphi_s>-C$, for some constant $C$ independent of $s$. Hence the upper bound of $\varphi_s$ can be obtained by standard maximum principle.
\end{remark}
\begin{remark}
We also remark that Dirichlet problem of complex Monge-Amp\`ere equation with semipositive reference form is considered by \cite{PS3}. In that note the authors study geodesic rays in
the space of K\"ahler potentials. In their case the right hand side is 0 and hence can be obtained without using pluripotential theory. 
\end{remark}
We do some preparations for the proof of Proposition \ref{C0}. The following comparison principle is well known.
\begin{lemma}\label{Comparison}
 Let $\mathcal U$ be as above and $\theta$ be a K\"ahler metric on $\mathcal U$, then for $u,v\in PSH(\theta)\cap L^{\infty}(\cU)$ satisfying $\lim_{\eta\rightarrow z}(u-v)(\eta)\ge 0$ for any $z\in\partial\cU$ we have
 $$\int_{\{u<v\}}(\theta+\ddbar v)^n\le\int_{\{u<v\}}(\theta+\ddbar u)^n.$$
 \end{lemma}
 We also introduce some  concepts in pluripotential theory.
 For a compact set $K\subset\subset \cU$ (here we still use $\cU$ to denote $\pi^{-1}\cU$ for simplicity), we define 
\begin{equation}
\begin{aligned}
U_{K,\cU}(x):&=sup\{u(x)|u\in PSH(\cU),u|_K=-1, -1\le u< 0\},\\
 U_{\theta,K,\cU}(x):&=sup\{u(x)|u\in \theta- PSH(\cU),u|_K=-1,-1\le u< 0\},\\
 Cap(K,\cU):&=sup\{\int_K(\ddbar u)^n, u\in PSH(\cU),-1<u<0\},\\ Cap_\theta(K,\cU):&=sup\{\int_K(\theta+\ddbar u)^n,u\in \theta -PSH(\cU),-1< u<0\}.\\
 \end{aligned}
\end{equation}
 
 The following two lemmas show that the capacity $Cap_{\theta}(K,\cU)$ can be computed by extremal function $U^*_{\theta,K}$ (For simplicity, we omit $\cU$). Similar results are proved on strongly pseudoconvex domain on $\mathbb C^n$ in \cite{BT} and on compact complex manifold in \cite{G-Z}. Since we haven't find a reference where the set-up is exactly the same as we are considering here, we  give a detailed proof. 
 \begin{lemma}\label{supp}
   Let $\theta$ be a K\"ahler form on $\cU$ and $K\subset\subset\cU$ be compact set. Then the measure $(\theta+\ddbar U^*_{\theta,K})^n$ is supported on $K \cup \{U^*_{\theta,K}=0\}$, where $U^*_{\theta,K}$ is the upper semi regularization of function $U_{\theta,K}$.
 \end{lemma}
 \begin{proof} Firstly, we prove that 
 $$U_{\theta,K}(x)=sup\{u(x)|u\in \theta- PSH(\cU)\cap C(\cU),u|_K=-1,-1\le u< 0\},$$
 where $C(\cU)$ is the set of continuous function on $\cU$. It suffices to prove that for any $x\in \cU$, $\epsilon\in (0,1)$ and any $\theta-$ PSH $(\cU)$ function $u$ satisfying $u|_K=-1$ and $-1\le u< 0$, there exists a continuous $\theta-$ PSH function $v$
such that $u(x)-\epsilon<v(x)<0, v|_K=-1$. For $\delta>0$ small,
let $$\cU_\delta:=\{x\in\cU,\rho>a-\delta \}.$$
We may assume $\rho-a\leq -1$ on $K$. If $\delta>0$ is choosen sufficiently small, then $\rho-a>-\frac{\epsilon}{2}$ on $\cU_{2\delta}$. Use the standard regularization of $\theta-$ PSH function for $u$ (see for example \cite[Thm 2]{BK}), there exists a smooth function $\theta-$ PSH function $q
\geq u$ defined on $\cU\setminus\cU_{\delta}$ satisfying $q-\epsilon<-\frac{\epsilon}{2}$ and  $q-\epsilon\leq -1$ on $K$.   Now the desired function $v$ can be constructed as follows: 
\begin{equation}
\left\{
\begin{array}{rcl}
&&v=\max\{q-\epsilon, \rho-a,-1\} \,\,\, on \,\,\,\, \cU\setminus \cU_\delta, \\
&&v=\rho-a  \,\,\,\,\,\,\quad\quad\quad\quad\quad\quad\quad on\,\, \cU_\delta.\\
\end{array}\right.
\end{equation}  It follows from Choquet's Lemma
that
there exists an increasing sequence of continuous $\theta$-PSH functions $\varphi_j$ such that $\varphi_j = - 1$ on $K$ , $\varphi_j<0$ on $\cU$, and $U_{\theta,K}^*= (\lim \varphi_j)^*$. Now let $x$ be a point not contained in $K \cup \{U^*_{\theta,K}=0\}$ and $B$ be an open ball containing $x$. Let $\tilde\varphi_j$ be the unique continuous functions  solving the Dirichlet problem, 
\begin{equation}
\left\{
\begin{array}{rcl}
&&\int_B(\theta +\ddbar\tilde\varphi_j)^n =0, \\
&&\tilde{\varphi_j}|_{\partial B}={\varphi_j}_{|\partial B}.
\end{array}\right.
\end{equation} 

  Notice that  $\Delta_\theta\tilde\varphi_j\geq -1, \tilde{\varphi_j}_{|\partial B}<0$. By maximum principle, if $B$ is chosen small enough, then $\tilde\varphi_j<0$ in $B$.  Also let $\tilde\varphi_j=\varphi_j$ on $\cU\setminus B$, then $\tilde\varphi_j \leq 0$ on $\cU$, while $\varphi_j= \tilde\varphi_j = -1$ on $K$. Therefore
$U_{\theta,K}^*=(\lim_j\nearrow\tilde\varphi_j)^*$, hence $\int_B(\theta+\ddbar U_{\theta,K}^*)^n=0$ in ball $B$, which proves our lemma.
 \end{proof}
\begin{lemma}\label{extremal}
Let $\theta$ be any K\"ahler form on $\cU$. Then for a compact set $K$ in $\cU$, we have $Cap_{\theta}(K,\cU)=\int_{K}(\theta+\ddbar U^*_{\theta,K})^n$, where $U^*_{\theta,K}$ is the upper semi regularization of the function $U_{\theta,K}$.
\end{lemma}
\begin{proof}
  By Lemma \ref{supp}, $(\theta+\ddbar U^*_{\theta,K})^n$ is supported on $K\cup \{U^*_{\theta,K=0}\}$ (In our setting, the set $\{ U^*_{\theta,K}=0\}$ could be larger than the set $\partial\cU$). And since $U^*_{\theta,K}$ itself is $\theta$- PSH with value between $-1$ and $0$. Hence $\int_K(\theta+\ddbar U^*_{\theta,K})^n\le Cap_\theta(K,\cU)$.
On the other hand, fixing a $\theta$- PSH function $u$ with $-1<u<0$, we have \begin{align*}\int_K(\theta+\ddbar u)^n&\le\int_{\{U^*_{\theta,K}<u\}}(\theta+\ddbar u)^n\\&\le\int_{\{U^*_{\theta,K}<u\}}(\theta+\ddbar U^*_{\theta,K})^n\\&=\int_K(\theta+\ddbar U^*_{\theta,K})^n.\end{align*}
The first inequality is due to the facts that $(\theta+\ddbar u)^n$ doesn't charge mass on the pluripolar $\{U_{\theta,K}\ne U^*_{\theta,K}\}$ and $K\subset \{U_{\theta,K}<u\}$. The second  inequality is due to comparison principle Lemma \ref{Comparison} and the third inequality is due to the facts that $\{U^*_{\theta,K}<u\} \cap \{U^*_{\theta,K}=0\}=\emptyset$ and $(\theta+\ddbar U^*_{\theta,K})^n$ is supported on $K\cup \{U^*_{\theta,K}=0\}$  .
 This  finishes the proof.
 \end{proof}
 
Now we return to the equation (\ref{scmp1}). Let
$$\theta_s:=\omega+s\theta.$$
For fixed $s$, we are interested in the set where $\varphi_s$ is small. We define:
 $$U(l):=\{\varphi_s<-l\},a(l):=Cap_{\theta_s}(U(l),\cU),b(l):=\int_{U(l)}(\theta_s+\ddbar \varphi_s)^n.$$
We should mention that $U(l),a(l),b(l)$ all depend on $s$. 

The following lemma roughly says that $Cap_{\theta_s}(U(l),\cU)$ can be controlled by $b(l+t)$. The argument given below is taken from \cite{Kol2}.
\begin{lemma}\label{C-M}
 Let $\varphi_s$ be the solution of equation of (\ref{scmp1}).
Then for any $0<t<1, l>2$, we have $t^nCap_{\theta_s}(U(l+t),\cU)\le \int_{U(l)}(\theta_s+\ddbar \varphi_s)^n$. 

 \end{lemma}
 \begin{proof}
Consider any compact regular set $K\subset U(l+t)$, the $\theta_s-PSH$ function $W:=\frac{1}{t}(\varphi_s+l)$, and the set $V:=\{W<U^*_{\theta_s,K}\}$. We can verify the inclusions $K\subset V\subset U(l)$.  
 Once we have the inclusions, we can apply Lemma \ref{Comparison} and Lemma \ref{extremal} to conclude:
 $$Cap_{\theta_s}(K,\cU)= \int_K(\theta_s+\ddbar U^*_{\theta_s,K})^n\le\int_V(\theta_s+\ddbar U^*_{\theta_s,K})^n\le\int_V(\theta_s+\ddbar W)^n$$$$\le t^{-n}\int_V(\theta_s+\ddbar \varphi_s)^n\le t^{-n}\int_{U(l)}(\theta_s+\ddbar \varphi_s)^n=t^{-n}b(l).$$ 
 Since $Cap_{\theta_s}(U(l+t),\cU)$ can be approximated by $Cap_{\theta_s}(K,\cU)$ (cf.), we finish the proof. 
 \end{proof}


We also want to show the Monge-Amp\`ere measure $\lambda(K):=\int_{K}(\theta_s+\ddbar\varphi_s)^n$ can by controlled by capacity.
\begin{lemma}\label{M-C}For any compact set $K\subset\subset\cU$, we have $\lambda(K)\le CCap^2_{\omega_s}(K,\cU)$ for some constant $C$ independent of $s$.
 
\end{lemma}

\begin{proof} By Remark \ref{E}, the $1$-$1$ forms  $\omega,\theta$ on $\cU$ can be choosen as the restriction of closed $1$-$1$ forms $\tilde\omega,\tilde\theta$ on some projective manifold $\bar Y$. Using equation (\ref{scmp1}), we know that $\lambda(K)$ is $L^p$ integrable 
 with respect to a fixed
measure for some $p>1$. Then it is standard that $\lambda(K)<CCap_{\omega+s\tilde\theta}(K,Y)^2$  for some $C$ independent of $s$ (cf. \cite[Cor 3.1]{FGS}). Since $\cU\subset Y$, it is easy to see from the definition that $Cap_{\omega+s\tilde\theta}(K,Y)\le Cap_{\theta_s}(K,\cU)$. This finishes the proof.

\end{proof}

\begin{lemma}
For any compact set $K\subset\cU$, $lim_{s\rightarrow 0}Cap_{\theta_s}(K,\cU)=Cap_{\omega}(K,\cU)$.
\begin{proof}
We know that $Cap_{\theta_s}(K,\cU)=\int_{\cU}(\theta_s+\ddbar U^*_{K,\theta_s})^n$. And by definition $U^*_{K,\theta_s}$ is decreasing when $s\rightarrow 0$. By \cite[Prop 1.2]{EGZ}, $U^*_{K,\theta_s}$ will decrease to a $\omega-$ PSH function $U$ with $$\int_{{K}}(\theta_s+\ddbar U^*_{K,\theta_s})^n\rightarrow\int_{{K}}(\omega+\ddbar U)^n.$$ The right hand side  by definition is smaller than  $Cap_{\omega}(K,\cU)$, hence we finish the proof.
\end{proof}
\end{lemma}

At last we need to show that the capacities have uniform decay independently of $s$.
\begin{proposition}\label{decay} Let $\varphi_s$ be the solution of equation (\ref{scmp1}), then $Cap_{\theta_{s}}(U(l+1),\cU)<C\frac{1}{l^{2n}}$ for some constant $C$ independent of $l$ and $\theta_s$.
\end{proposition}

\begin{proof}
 Our key observation is that on $\cU$, $\omega$ can be represented by $\ddbar\rho$, so
 \begin{equation}\label{aa}
 Cap_\omega(K,\cU)\le ACap(K,\cU),\end{equation} where $A$ only depends on  $|\rho|_{L^\infty(\cU)}$, which is bounded. This enables us to compare $Cap_{\theta_s}(K,\cU)$ with $Cap(K,\cU)$ uniformly.
 Now we conclude that:
 \begin{equation}
     \begin{aligned}
 Cap_{\theta_s}(U(l+1),\cU)&\le\int_{\bar U(l)}(\theta_s+\ddbar\varphi_s)^n=\lambda(\bar U(l))\le ACap^2_\omega(\bar U(l),\cU)\\
 &\le ACap^2(\bar U(l),\cU)\le ACap^2_{\frac{\theta_{s}}{l-\eta}}(\bar U(l),\cU)\\&=A\left(\int_{\bar U(l)}(\frac{\theta_s}{l-\eta}+\ddbar U^*_{\bar U(l),\frac{\theta_s}{l-\eta}})^n\right)^2,
 \end{aligned}
 \end{equation}
 where $l>>\eta>0$. Here we have used formula (\ref{aa}), Lemma  \ref{C-M} and Lemma \ref{M-C} and also constant $A$ might change from line to line.  Note $\varphi_s=0$ on $\partial\cU$ and $\frac{\varphi_s}{l-\eta}<-1$ on $\bar U(l)$,  by comparison principle we get 
 \begin{align*}\int_{\bar U(l)}(\frac{\theta_{s}}{l-\eta}+\ddbar U^*_{ \bar U(l),\frac{\theta_s}{l-\eta}})^n&\le\int_{\bar U(l)}(\frac{\theta_s}{l-\eta}+\ddbar \frac{\varphi_s}{l-\eta})^n\\&\le \frac{1}{(l-\eta)^n}\int_{\cU}(\theta_{s}+\ddbar \varphi_s)^n\\&\le C\frac{1}{(l-\eta)^n}.\end{align*}
 Now let $\eta\rightarrow 0$, then we finish the proof.

\end{proof}

The following lemma is well-known and its proof can be found e.g. in \cite{ EGZ}. 
\begin{lemma}\label{lemma 3.4} Let $F:[0,\infty)\to [0,\infty)$ be a non-increasing right-continuous function satisfying
$\lim_{l\rightarrow \infty} F(l) =0$. 
If there exist $\alpha, A>0$ such that for all $s>0$ and $0\le r\le 1$, 
\begin{equation*}\label{eqn:assumption1} rF(l+r) \leq A \left( F(l)\right)^{1+\alpha},\end{equation*}
then there exists $S = S(l_0, \alpha, A)$ such that $$F(l) = 0,$$ for all $l\ge S$, where $l_0$ is the smallest $l$ satisfying $\left( F(l)\right)^\alpha \leq (2A)^{-1}$.
\end{lemma}

\begin{proof}[Proof of Proposition \ref{C0}.]

Define for each fixed $l$ large, $$F(l):=Cap_{\theta_s}(U(l),\cU) ^{1/n}.$$
By Lemma \ref{C-M} and Lemma \ref{M-C} applied to the function $\varphi_s$, we have
$$r F(l+r)\le A F(l)^2,\quad \textnormal{ for ~ all} ~ r \in [0,1], \, l>2,$$
for some uniform constant $A>0$ independent of $r\in (0,1]$. 
Proposition \ref{decay} implies that $\lim_{l\to \infty} F(l) = 0$ and the $l_0$ in Lemma \ref{lemma 3.4} can be taken as less than $(2 A C)^q$, which is a uniform constant. It follows from Lemma \ref{lemma 3.4} that $F(l) = 0$ for all $l>S$, where $S \le 2 + l_0$. On the other hand, 
if $Cap_{\omega} \{\varphi_s<-l \} = 0$, by Lemma \ref{C-M} , we have the integral $b(l)=0$. Hence the set $\{\varphi_s <-l\} = \emptyset$. Thus $\inf_\cU (\varphi_s )\ge - S$. Thus we finish the proof of Proposition \ref{C0}.
\end{proof}

\textbf{Step 2:} In this step, with the uniform $C^0$ estimate of $\varphi_s$ in hand, from Lemma \ref{BC1} to Lemma \ref{differ},  we establish high order estimates for $\varphi_s$. These lemmas are taken from \cite{DFS}, which in turn is based on many important boundary and global estimates for Monge-Amp\`ere equation from \cite{Blo,CKNS,PSS1,PS,PS1}. 

We first fix notations. We will denote the covariant derivative of $\theta$ by $\nabla$. We also let $D$ be the effective divisor from Lemma \ref{lem:kodaira}, such that $\mathrm{Supp}(D) = \mathrm{Supp}(F)\cup \mathrm{Supp}(E)$ and 
\begin{equation}\label{eq:kodaira}
\omega+\ddbar\log h_D > \be\theta,
\end{equation} for some $\be>0$. Note that by the support condition, if $\sigma_D$ is a defining section of $D$, then there exists a uniform constant $C$ such that 
\begin{equation}\label{eq:estforsections}
|\sigma_F|^{-2},|\sigma_E|_{h_E}^{-2}\leq C|\sigma_D|^{-2l},
\end{equation}  for some $l\in \mathbb{N}$.

Next, from Lemma \ref{BC1} to Lemma \ref{2ndem}, we obtain global $C^1$ and $C^2$ estimates of $\varphi_s$ with suitable barrier functions.
\begin{lemma}\cite[Lemma 2.3]{DFS}\label{BC1} There exists a constant $C>0$ such that for all $0<s<1$,   $$|\na\varphi_{ s}|_{\partial\cU}\le C.$$ 
\end{lemma}

\begin{lemma}\cite[Prop 2.2]{DFS}\label{GC1}
Let $\varphi_{s}$ be the solution of equation  \eqref{scmp11}, There exist $N,C>0$ such that for all $0<s<1$,   $$|\na\varphi_{ s}|^2|\sigma_D|^{2N}_{h_D}\le C.$$  
\end{lemma}

\begin{lemma}\cite[Lemma 2.4]{DFS}\label{BC2}
Let $\varphi_{ s}$ be the solution of equation  \eqref{scmp11}, then there exist constant $C>0$ such that for all $0<s<1$,    $$|\na_\theta^2\varphi_{ s}|_{\partial\cU}\le C.$$
\end{lemma}

Next, we show second order estimates with bounds from suitable barrier functions.  We include it for reader's convenience and also for later purpose. (See also \cite{ST2}) 

\begin{lemma}\label{2ndem} There exist constants $N,C>0$ such that for all $0<s<1$,  
\begin{equation}
\sup_{\cU} \left(  |\sigma|^{N}_{h_D} \right)  |\Delta_\theta \varphi_{ s} | \leq C,
\end{equation}
where $\Delta_\theta$ is the Laplace operator with respect to the K\"ahler metric $\theta$.
\end{lemma}
\begin{proof} 

Let $\omega' =  \omega + s\theta + \ddbar \varphi_{s}$. Then we consider the quantity $$H= \log tr_\theta (\omega') - B\varphi_{ s} +B \log |\sigma|^2_{h_D},$$
for some large constant $B>2$. By the $C^0$ estimate of $\varphi_{ s}$, $H$ is bounded from above in $\cU$. Standard calculations show that 
\begin{equation}\label{eq:aubin-yau-c2}
\Delta ' \log tr_{\theta}{\omega'} \geq -Ctr_{\omega'}{\theta} - \frac{tr_{\theta}Ric(\omega')}{tr_{\theta}(\omega')}. 
\end{equation} From equations \eqref{scmp11}, \eqref{eq:estforsections} and the elementary observation that $$\ddbar\log(f+s) = \frac{f}{f+s}\ddbar\log f + s\frac{\sqrt{-1}\partial f\wedge \overline{\partial} f}{f(f+s)}$$ holds for any smooth non-negative function $f$, it is easy to see that $$-tr_{\theta}Ric(\omega') \geq \frac{-C}{|\sigma_D|_{h_D}^{2l}},$$ for some constant $C>0$ independent of $s$.  Together with \eqref{eq:aubin-yau-c2} and our choice of $D$ (cf.\ \eqref{eq:kodaira})  we see that 
\begin{align*}
\Delta'H &\geq -Ctr_{\omega'}{\theta}  - \frac{C}{|\sigma_D|_{h_D}^{2l}tr_{\theta}(\omega')} + Btr_{\omega'}(\ddbar \log h_D - \ddbar\varphi_s)\\
&\geq (B\be - C)tr_{\omega'}{\theta} - \frac{C}{|\sigma_D|_{h_D}^{2l}tr_{\theta}(\omega')} - Bn\\
&\geq tr_{\omega'}{\theta} - \frac{C}{|\sigma_D|_{h_D}^{2l}tr_{\theta}(\omega')} - Bn \hspace{0.5in} (\text{if } B>>1)\\ 
&\geq (tr_\theta (\omega'))^{\frac{1}{n-1}}(\frac{\theta^n}{\omega'^n})^{\frac{1}{n-1}} - \frac{C}{|\sigma_D|_{h_D}^{2l}tr_{\theta}(\omega')} - Bn\\
&\geq  (tr_\theta (\omega'))^{\frac{1}{n-1}}|\sigma_D|_{h_D}^{2\al} - \frac{C}{|\sigma_D|_{h_D}^{2l}tr_{\theta}(\omega')} - Bn,
\end{align*}
for some constant $\al$ independent of $s$ and $B>>1$ so that $B\be > C+1$ in line three. 

%
%
%
%
%
%
By Lemma \ref{BC2}, it suffices to assume that $ H$ obtains maximum at a point $p\in \cU$. Moreover, since $H$ goes to $-\infty$ on $\mathrm{Supp}(D)$, clearly $p\notin \mathrm{Supp}(D)$.  From the maximum principle it follows that there is an integer $k$ (depending on $\al,l$ and $n$) such that $$|\sigma_D|_{h_D}^{2k}(tr_\theta\omega')(p)\leq C.$$ It then follows easily (possibly by choosing a $B>>k$) that for any $x\in \cU\setminus\mathrm{Supp}(D)$ we have
\begin{align*}
|\sigma_D|^{2N}_{h_D}(tr_{\theta}\omega')(x) \leq C,
\end{align*} for some $N$ large enough. This gives an upper bound on $|\sigma_D|^{2N}_{h_D}|\Delta_\theta\varphi_s|$. On the other hand we also have the trivial lower bound that $\Delta_\theta\varphi_s > -tr_\theta\omega - sn$. This completes the proof of the Lemma.   
%
%

\end{proof}

The following lemma on local higher order regularity of $\varphi_{ s}$ is established by the standard linear elliptic theory after applying Lemma \ref{2ndem} and linearizing the complex Monge-Amp\`ere equation \eqref{scmp11}. We state it without proof.

\begin{lemma} \label{lem:uniform-smooth-estimates}
For  any compact  $K \subset\subset (\cU \setminus p) $, and any natural number $k\in \mathbb{N}$, there exists a constant $C=C( k, K)>0$ such that for any $0<s<1$ 
$$ || \varphi_{ s} ||_{C^k(K)} \leq C. $$

\end{lemma}

Before we let $s \rightarrow 0$, we obtain a uniform bounds on the `$s$' derivative of $\varphi_s$.

\begin{lemma}\cite[Lemma 2.7]{DFS} \label{differ}
For  any compact  $K \subset\subset (\cU \setminus \mathrm{Supp}(D)) $ , there exists a constant $C=C( K)>0$ such that for any $0<s<1$, we have
\begin{equation}\label{1derest}
\sup_{K} \left|  \frac{\partial \varphi_{ s}}{\partial s} \right| \leq C. 
 \end{equation}
 
\end{lemma}

A corollary of the a priori estimates established is the following generalization of Lemma \ref{blowup}. 
\begin{corollary}\label{blowup1} Let $(X,p)$  be  an isolated Klt  singularity, then there exists a solution to the equation (\ref{scmp11}) for any fixed $s>0$.
\end{corollary}
\begin{proof}
For fixed $s>0,t\in [0,1]$, consider
 \begin{equation}\label{scmp111}
\left\{
\begin{array}{rcl}
&&(\omega + s\theta + \ddbar \varphi_{s,t})^n = ((1-t)+t(g\frac{\Omega_{Y}}{(\chi+s\theta)^n}))(\chi+s\theta)^n, \\
&&\varphi_{s,t}|_{\partial\Omega}=0.
\end{array}\right.
\end{equation} 
By Proposition \ref{C0}, $\varphi_{s,t}$ is uniformly bounded independently of $t$. Then by the same argument as in \cite[Thm 1.1]{DFS} (or step 2 above where we obtained the higher order estimates), we can finish the proof by the standard continuity method.
\end{proof}
Now we are able to prove our Theorem  (\ref{main1}). 

\begin{proof}[Proof of Theorem \ref{main1}:]
Combine Proposition \ref{C0}, Lemma \ref{BC1}, Lemma \ref{GC1}, Lemma \ref{BC2}, Lemma \ref{2ndem}, Lemma \ref{lem:uniform-smooth-estimates} Lemma \ref{differ} and take $\varphi:=\lim_{s\to 0}\varphi_s$ to complete the proof.
\end{proof}

 We remark that, if we change the right hand side of equation (\ref{ocmp1}) to $e^{\varphi}\Omega_X$, a useful $C^0$ estimate without using pluripotential theory is obtained in \cite{DFS} (see also \cite{DGG,SSW}). Using this new estimate, we give a simplified proof of Theorem \ref{EGZ} without using pluripotential theory but with weaker $C^0$ estimate.

\begin{proposition}\label{b}
 Suppose $X$ is canonical polarized  variety with Klt singularity, consider 
 
 \begin{equation}\label{a} 
(\omega+ \ddbar \varphi)^n = e^{ \varphi} \Omega_X^n,
\end{equation}
where $\Omega_X$ is a volume form satisfying $\ddbar\log\Omega_X^n=\omega>0$. Then there is a solution $\varphi\in C^\infty(X_{reg})$  to the above equation on smooth locus $X$.
\end{proposition}
\begin{proof}
We perturb the equation (\ref{a}) as we perturb equation (\ref{scmp11}),
$$(\omega+ s\theta + \ddbar \varphi_s)^n =  (|\sigma_E|^2_{h_E}+ s)  (|\sigma_F|^2_{h_F}+ s)^{-1}\Omega_Y.$$
The following  $C^0$ estimate from \cite{DFS} is the key of our simplified proof of Theorem \ref{EGZ}. 
\begin{lemma}\cite[Prop 2.1]{DFS}\label{prop:c0} 
 There exists an effective divisor $D$ with $\mathrm{Supp}(D) =  \mathrm{Supp}(F)\cup \mathrm{Supp}(E)$ and constant $C>0$ such that for all $$C \geq \varphi_s \geq -2n\log \Big(-\log|\sigma_D|_{h_D}^2\Big) - C.$$ Here $\sigma_D$ is a defining section of $D$ and $h_D$ is any hermitian metric on the line bundle corresponding to $D$.
 \end{lemma}
 
With this $C^0$ estimate for $\varphi_s$ and also Lemma \ref{2ndem} in hand, we can use the same argument as in proof of Theorem \ref{main1} to complete proof. \end{proof}

\begin{remark}
The result we obtained in the new proof is weaker than the result in \cite{EGZ} in the sense that our estimate doesn't imply $\varphi\in L^\infty_{loc}(X_{reg})$.
\end{remark}

At last, we show that there are singular solution to the M\"onge-Amp\`ere equation on a strongly pseudoconvex domain $\Omega$ contained in $\mathbb C^n$.

\begin{proof}[Proof of Corollary \ref{Singularsolution}]
Denote the  blowing up of the origin by $\pi: Bl_{0}\mathbb C^n\to \mathbb C^n$ and pull back the equation (\ref{Singular}) to  $\pi^{-1}\Omega$. Fix a hermitian metric $h_E$ on the line bundle associated to the exceptional divisor $E$, then  $$|\log(\sum_{i=1}^n|z_i|^2)-\log|E|^2_{h_E}|=\mathcal O(1)$$ To construct a solution with desired property stated in the corollary, we only need to construct a solution $\varphi$ to equation (\ref{Singular}) on $\pi^{-1}\Omega$ with $|\varphi-s\log|E|^2_{h_E}|=\mathcal O(1)$. Denote the defining function of domain $\Omega$ by $\rho$, i.e., $\Omega=\{\rho<0\}$ and $\ddbar\rho>0$. By choosing a large constant $A$, we have by Lemma \ref{lem:kodaira} that
$$\omega:=\ddbar(A\pi^*\rho)+s\ddbar\log h_E>0.$$ Now let $\phi=\varphi-s\log|E|^2_{h_E}-A\rho$ and rewrite the equation (\ref{Singular}) as 
 \begin{equation}\label{S}
\left\{
\begin{array}{rcl}
&&( \omega+ \ddbar \phi)^n = e^{f+\lambda(\phi+A\rho+s\log|E|^2_{h_E})}\pi^*(V\wedge\bar V),\\
&&\phi|_{\partial\Omega}=\psi-s\log|E|^2_{h_E}.
\end{array}\right.
\end{equation}
The above equation holds on $\pi^{-1}(\Omega)\setminus E$.  Now we may perturb the equation (\ref{S}) as we did in equation (\ref{scmp}). More precisely, for $t>0$ small, consider the following family of equations: \begin{equation}
\left\{
\begin{array}{rcl}
&&( \omega+ \ddbar \phi_t)^n =e^F (|E|^{2(\lambda s+n-1)}_{h_E}+t)dW,\\
&&\phi_t|_{\partial\Omega}=\psi-s\log|E|^2_{h_E},
\end{array}\right.
\end{equation}
where $F:=f+\lambda(\phi_t+A\rho)$ and $dW$ is a strictly positive volume form on $\pi^{-1}(\Omega)$ . Note that in the above equation, $\omega$ is  K\"ahler, $(\psi-s\log|E|^2_{h_E})|_{\partial\Omega}$ is smooth and that the term $|E|^{2(\lambda s+n-1)}_{h_E}$ is $L^p,p>1$ integrable when $s>0$. 
Now by the same argument as in the proof of Theorem \ref{main1}, we conclude that equation (\ref{S}) admits a  bounded solution $|\phi|<C$ on $\pi^{-1}(\Omega)\setminus E$ as the limit of $\phi_t$, which completes the proof.
\end{proof}

\begin{remark}
The above corollary can be easily generalized to the case when the solution $\varphi=\sum_{j=1}^{m}s_j\log(\sum_{i=1}^n|z^j_i|^2)+\mathcal O(1)$, where $s_j>0$ and $z^j_i,1\leq j\leq m, 1\leq i\leq n$ are local holomorphic coordinates at finite points $p_j\in\Omega$. 
\end{remark}

\begin{remark}
When preparing this paper, the author learned that similar ideas have been used to study Plurigreen function with prescribed singularity in \cite{PS1}. Moreover, combining \cite[Lemma 1]{PS1} and our argument here,  Corollary \ref{Singular} can also be generalized to the case that solution $\varphi=s\log(\sum^n_{i=1}|f_i|^2)+\mathcal O(1)$, where $f_i$ are holomorphic functions in a neighbourhood of origin $\{o\}$ with $\{o\}$ as their only common zero. We refer the reader to the very interesting result of \cite{PS1} concerning Plurigreen functions.
\end{remark}

\end{document}